\documentclass[11pt]{amsart}
\usepackage[utf8]{inputenc}
\usepackage{graphicx,color,verbatim}
\usepackage{amscd,amsmath,amsfonts,amssymb}
\usepackage{mathrsfs}
\newtheorem{theorem}{Theorem}
\newtheorem{lemma}[theorem]{Lemma}

\newtheorem{remark}[theorem]{Remark}

\newtheorem{example}{Example}

\numberwithin{equation}{section}
\newcommand{\R}{\mathbb R}
\newcommand{\N}{\mathbb N}

\begin{document}
\hfill\today\bigskip

\title[]{Existence results for a critical fractional equation}

%\thanks{The authors are members of the Gruppo Nazionale per l'Analisi Matematica, la Probabilit\`{a} e le loro Applicazioni (GNAMPA) of the Istituto Nazionale di Alta Matematica (INdAM)}

\author[H. Hajaiej]{Hichem Hajaiej}
\address{New York University, Shanghai\\
1555 Century Avenue\\
Pudong New District, Shanghai 200122, China}
\email{\tt hichem.hajaiej@gmail.com}

\author[G. Molica Bisci]{Giovanni Molica Bisci}
\address{Dipartimento P.A.U. \\
Università degli Studi Mediterranea di Reggio Calabria\\
Salita Melissari - Feo di Vito\\
89124 Reggio Calabria, Italy}
\email{\tt gmolica@unirc.it}

\author[L. Vilasi]{Luca Vilasi}
\address{Department of Mathematical and Computer Sciences, Physical Sciences and Earth Sciences\\
University of Messina\\
Viale F. Stagno d'Alcontres, 31\\
98166 Messina, Italy}
\email{\tt lvilasi@unime.it}

\keywords{Critical nonlinearities, singular nonlinearities, integro-differential operators, fractional Laplacian, existence \\
\phantom{aa} 2010 AMS Subject Classification: Primary: 49J35, 35A15, 35S15;
Secondary: 47G20, 45G05.}

\begin{abstract}
We are concerned with existence results for a critical problem of Br\'{e}zis-Nirenberg--type driven by an integro-differential operator of fractional nature. The latter includes, for a specific choice of the kernel, the usual fractional Laplacian. Under mild assumptions on the subcritical part of the nonlinearity, we provide first the existence of one weak solution through direct minimization of the energy in a small ball of a certain fractional Sobolev space. This approach remains still valid when adding small singular terms. We finally show that for appropriate choices of the parameters involved the mountain-pass approach is also applicable and yields another existence result.
\end{abstract}

\maketitle

%\tableofcontents

%%%%%%%%%%%%%%%%%%%%%%%%%%%%%%%%%%%%%%%%%%%%%
\section{Introduction}\label{sec:introduzione}
In the present paper we explore the existence of weak solutions to the following Br\'{e}zis-Nirenberg--type problem:
\begin{equation}\tag{$P_{\lambda,\mu}$}\label{problema}
\left\{
\begin{array}{ll}
-\mathcal L_K u = \mu|u|^{2^*-2}u +\lambda g(u) & \mbox{ in } \Omega\\
u=0 & \mbox{ in } \R^n\setminus \Omega,
\end{array} \right.
\end{equation}
where $n>2s$, $s\in(0,1)$, $\Omega\subset\R^n$ is open, bounded and with Lipschitz boundary, $\lambda,\mu$ are positive real parameters,
\begin{equation}\label{2star}
2^*:=\frac{2n}{n-2s}
\end{equation}
is the fractional critical Sobolev exponent and $g:\R\to\R$ is continuous with subcritical growth. The leading term $\mathcal L_K$ is a nonlocal fractional operator of order $s$ defined by
$$
\mathcal L_K u(x):= \int_{\R^n}(u(x+y)+u(x-y)-2u(x))K(y)dy, \quad x\in \R^n,
$$
where the function $K:\R^n\setminus\{0\}\rightarrow (0,+\infty)$ satisfies:
\begin{itemize}
\item[$(k_1)$] $w K\in L^1(\R^n)$, {\it with} $w(x):=\min \{|x|^2, 1\}$;
\item[$(k_2)$] \textit{there exists} $k_0>0$ \textit{such that}
$$
K(x)\geq \frac{k_0}{|x|^{n+2s}},
$$
\mbox{ \it for any } $x\in \R^n \setminus\{0\}$.
\end{itemize}

The paradigm of the above $K$ is given by the singular kernel $K(x):=|x|^{-(n+2s)}$; in this case $\mathcal L_K$ reduces to the fractional Laplace operator defined, pointwise and up to normalization constants, by
$$
(-\Delta)^s u(x):=-\int_{\R^n}\frac{u(x+y)+u(x-y)-2u(x)}{|y|^{n+2s}} dy, \quad x\in \R^n.
$$

Nonlinear problems like \eqref{problema} can be traced back to the well-known and pioneering \cite{brenir1983positive}, where Br\'{e}zis and Nirenberg dealt with the following critical elliptic problem
\begin{equation}\label{1}
\left\{
\begin{array}{ll}
-\Delta u-\lambda u=|u|^{2_*-2}u & \mbox{ in } \Omega\\
u=0 & \mbox{ on } \partial \Omega,
\end{array} \right.
\end{equation}
being $\Omega$ a smooth bounded domain of $\R^n$, $\lambda>0$ and $2_*:=2n/(n-2)$, $n>2$. By guaranteeing a local Palais-Smale condition for the energy associated with \eqref{1} they proved variationally the existence of solutions for $\lambda$ running in appropriate subintervals of $(0,\lambda_1)$, $\lambda_1$ being the first eigenvalue of $-\Delta$ in $H^1_0(\Omega)$. Ever since such a problem has been the object of in-depth studies, also for its close relationship with intriguing topics arising in differential geometry (see for instance the Yamabe problem \cite{yam1960on}), and has been extended to more general leading elliptic operators (cf. \cite{capforpal1985an,garper1991multiplicity,zou2013on} and the references enclosed).

Over the recent years problems à la Br\'{e}zis-Nirenberg have made their way into the nonlocal framework as well (see \cite{mawmol2016a,mospersquyan2016the,serval2013a,serval2015the,tan2011the}). In \cite{serval2015the} Servadei and Valdinoci provided the existence of one non-trivial solution to the problem
\begin{equation}\label{serval}
\left\{
\begin{array}{ll}
-\mathcal L_K u -\lambda u = |u|^{2^*-2}u + f(x,u) & \mbox{ in } \Omega\\
u=0 & \mbox{ in } \R^n\setminus \Omega,
\end{array} \right.
\end{equation}
with $f$ perturbation of lower order of the critical nonlinearity $|u|^{2^*-2}u$, extending all the variational apparatus of \cite{brenir1983positive} to the fractional context. In this paper it is crucial the use of a specific functional setting to prove that the energy satisfies the Palais-Smale condition under a certain threshold related to the best fractional Sobolev constant of the embedding $X_0\hookrightarrow L^{2^*}(\R^n)$ (see Section \ref{funcframework} below for the definition of the space $X_0$).

This approach can be successfully adopted for problem \eqref{problema} as well, with suitable modifications and provided that the parameter $\mu$ coupled with the critical term is small, and is briefly reported in the final Section \ref{mountainpass}. We wish to show, however, that existence results for the problem under exam may be obtained via a completely different method, i.e. by means of direct minimization on small balls of the space $X_0$. This is made possible by proving a local weak lower semicontinuity property of the energy for all $\mu>0$ and sufficiently small $\lambda$. The assumptions required on the subcritical nonlinearity $g$ are quite general and mild. The main trick consists in choosing $\lambda$ inside the range of a suitable rational function which involves, in particular, the sharp constant of the embedding $X_0\hookrightarrow L^{2^*}(\R^n)$.

When $\mathcal L_K$ degenerates into the fractional Laplacian, a particular case of our result can be stated as follows. Denote
\begin{equation}\label{X0s}
X_{0,s}:=\{u\in H^s(\R^n) : u=0 \mbox{ a.e. in } \R^n\setminus\Omega\},
\end{equation}
where $H^s(\R^n)$ is the usual fractional Sobolev space of order $s$, and put
\begin{align*}
\left\| u\right\|_{X_{0,s}}&:=\left(\int_{\R^n}\frac{|u(x)-u(y)|^2}{|x-y|^{n+2s}}dxdy \right)^{1/2},\\
c_{2^*_s}&:=\sup_{u\in X_{0,s}\setminus\{0\}}\frac{\left\| u\right\|_{L^{2^*}(\R^n)}}{\left\| u\right\|_{X_{0,s}}}
\end{align*}
(see Remark \ref{Costante} for further details).

\begin{theorem}\label{casoparticolare}
Let $g:\R\to\R$ be a continuous function for which there exist $a_1,a_2>0$ and $q\in[1,2^*)$ such that
\begin{equation}
|g(t)|\leq a_1 + a_2|t|^{q-1} \quad \mbox{ for all } t\in\R,
\end{equation}
and satisfying
\begin{equation}
\liminf_{t\to 0^+}\frac{\displaystyle\int_0^t g(\tau)d\tau}{t^2} = +\infty.
\end{equation}
If $h:[0,+\infty)\to\R$ is the function defined by
$$
h(t):=\frac{t- c_{2^*_s}^{2^*} t^{2^*-1}}{a_1 c_{2^*_s}|\Omega|^\frac{2^*-1}{2^*} + a_2 c_{2^*_s}^q |\Omega|^\frac{2^*-q}{2^*} t^{q-1}} \quad  \mbox{ for all } t\geq 0,
$$
then there exists an open interval $\displaystyle\Lambda\subseteq\left(0,\max_{[0,+\infty)}h \right)$
such that, for every $\lambda\in\Lambda$, the problem
\begin{equation}\tag{$P_{\lambda}$}\label{problemamodello}
\left\{
\begin{array}{ll}
(-\Delta)^s u = |u|^{2^*-2}u +\lambda g(u) & \mbox{ in } \Omega\\
u=0 & \mbox{ in } \R^n\setminus \Omega,
\end{array} \right.
\end{equation}
admits a non-trivial weak solution in $X_{0,s}$.
\end{theorem}

The existence of a local minimum continues to be guaranteed if we add a singular term to the perturbation $g$, namely if we look for positive solutions to the problem
\begin{equation}\tag{$\widetilde{P}_{\lambda,\mu}$}\label{problemasingolareintro}
	\left\{
	\begin{array}{ll}
		-\mathcal L_K u = \mu u^{2^*-1} + \lambda \left( u^{r-1} + g(u)\right)  & \mbox{ in } \Omega\\
		u=0 & \mbox{ in } \R^n\setminus \Omega,
	\end{array} \right.
\end{equation}
where $r\in(0,1)$ and $g:[0,+\infty)\to [0,+\infty)$ is as usual continuous and subcritical. In such a case the assumptions on $g$ may further be relaxed, the presence of the singularity being sufficient to ensure that the minimum determined is not trivial.

In this regard an application of our existence theorem reads as follows:
\begin{theorem}\label{casoparticolaresingolare}
	Let $g:[0,+\infty)\to[0,+\infty)$ be a continuous function for which there exist $a_1,a_2>0$ and $q\in[1,2^*)$ such that
	\begin{equation}
	g(t)\leq a_1 + a_2t^{q-1} \quad \mbox{ for all } t\geq 0.
	\end{equation}
	If $r\in(0,1)$ and $k:[0,+\infty)\to\R$ is the function defined by
	$$
	k(t):= \frac{t^{2-r}- c_{2^*_s}^{2^*} t^{2^*-r}}{c_{2^*_s}^r |\Omega|^\frac{2^*-r}{2^*} + a_1 c_{2^*_s}|\Omega|^\frac{2^*-1}{2^*} t^{1-r} + a_2 c_{2^*_s}^q |\Omega|^\frac{2^*-q}{2^*} t^{q-r}} \quad  \mbox{ for all } t\geq 0,
	$$
	then there exists an open interval $\displaystyle\Lambda\subseteq\left(0,\max_{[0,+\infty)}k \right)$
	such that, for every $\lambda\in\Lambda$, the problem
	\begin{equation}\tag{$\widetilde{P}_{\lambda}$}\label{problemamodellosingolare}
	\left\{
	\begin{array}{ll}
	(-\Delta)^s u = u^{2^*-1} + \lambda \left(  u^{r-1} + g(u)\right)  & \mbox{ in } \Omega\\
	u>0 & \mbox{ in } \Omega\\
	u=0 & \mbox{ in } \R^n\setminus \Omega,
	\end{array} \right.
	\end{equation}
	admits a weak solution in $X_{0,s}$.
\end{theorem}

We point out that this kind of approach has been already fruitfully used in literature for dealing with nonlinear critical problems; we mention the papers \cite{farfar2015a,squ2004two} for equations driven by $p$-Laplacian type operators and \cite{molrep2016yamabe} for equations set on Carnot groups. Furthermore we refer to the reading of \cite{autfispuc2015stationary,fismolser2016bifurcation,fispuc2016on,haj2016a,haj2016symmetry,haj2013on} as well as \cite{kuuminsir2015nonlocal,molser2015lower,pucsal2016critical,ser2013the,serval2015fractional} for existence and regularity results to nonlocal critical problems handled with different tools.

The paper has the following structure: in Section \ref{funcframework} we illustrate the functional setting of the problem together with its weak formulation; Sections \ref{directminimization} and \ref{mountainpass} are devoted to the existence results obtained via direct minimization and mountain-pass procedure, respectively.

%%%%%%%%%%%%%%%%%%%%%%%%%%%%%%%%%%%%%%%
\section{Functional framework}\label{funcframework}
For a detailed description of the functional spaces which arise in the variational formulation of \eqref{problema} and, in general, for a good survey on fractional Sobolev spaces, we refer to the papers \cite{dinpalval2012hitchhiker,molradser2016variational} and the rich bibliography therein.
%while for more details on $X$ and $X_0$ we refer to \cite{ser2013the,serval2012mountain,serval2013lewy, serval2013variational,serval2015the}.

%In order to describe the functional set-up of problem \eqref{problema},
We start with introducing $X$ as the linear space of all Lebesgue measurable functions from $\R^n$ to $\R$ such that the restriction to $\Omega$ of any function $u$ in $X$ belongs to $L^2(\Omega)$ and the map
$$
(x,y)\mapsto (u(x)-u(y))\sqrt{K(x-y)} \in L^2(Q,dxdy),
$$
where $Q:=\R^{2n} \setminus (\Omega^c\times\Omega^c)$ and $\Omega^c:=\R^n\setminus\Omega$.

Denote by $X_0$ the following linear subspace of $X$,
$$
X_0:=\{u\in X: u=0 \mbox{ a.e. in } \Omega^c\}.
$$
Both $X$ and $X_0$ are non-empty since, for instance, $C^2_0 (\Omega)\subseteq X_0$.
%(see \cite[Lemma 11]{serval2013lewy}).
It is not difficult to realize that
\begin{equation}\label{normaX}
\|u\|_X:=\|u\|_{L^2(\Omega)} + \left(\int_Q |u(x)-u(y)|^2 K(x-y) dxdy\right)^{1/2}
\end{equation}
defines a norm on $X$ and that
\begin{equation}\label{normaX0}
X_0\ni u\mapsto \|u\|_{X_0}:=\left(\int_Q|u(x)-u(y)|^2K(x-y)dxdy\right)^{1/2}
\end{equation}
represents a norm on $X_0$ equivalent to (\ref{normaX}).
%(see \cite{serval2012mountain}).
The space $\left(X_0, \|\cdot\|_{X_0}\right)$ turns out to be a Hilbert space with scalar product
\begin{equation}\label{prodscal}
\langle u,v\rangle_{X_0}:=\int_Q \left(u(x)-u(y)\right) \left(v(x)-v(y)\right) K(x-y)dxdy.
\end{equation}
%see \cite[Lemma 7]{serval2012mountain}.
In the sequel we will employ for brevity the plain symbols $\left\|\cdot\right\|$ and $\left\langle \cdot,\cdot\right\rangle$ to refer to $\left\|\cdot\right\|_{X_0}$ and $\left\langle \cdot,\cdot\right\rangle_{X_0}$, respectively. Moreover the open (respectively closed) ball centered at $u\in X_0$ of radius $r>0$ will be denoted by $B(u,r)$ (respectively, $B_c(u,r)$), and the sphere $\{u\in X_0: \left\| u\right\|=r\}$ by $\partial B(u,r)$.

In the model case $K(x):=|x|^{-(n+2s)}$, the space $X_0$ can be characterized as follows
$$
X_0=X_{0,s}=\big\{u\in H^s(\R^n) : u=0 \mbox{ a.e. in }  \Omega^c\big\},
$$
%(see \cite[Lemma~7(b)]{serval2015the}),
where $H^s(\R^n)$ denotes the usual fractional Sobolev space endowed with the so-called Gagliardo norm (not equivalent to (\ref{normaX})):
$$
\|u\|_{H^s(\R^n)}=\|u\|_{L^2(\R^n)}+
\left(\int_{\R^{2n}}\frac{|u(x)-u(y)|^2}{|x-y|^{n+2s}} dxdy\right)^{1/2}.
$$

Under assumptions $(k_1)$ and $(k_2)$ the space $X_0$ embeds in the canonical Lebesgue spaces; more specifically, the embedding $X_0\hookrightarrow L^p(\R^n)$ is continuous for any $p\in [1,2^*]$, while it is compact whenever $p\in [1,2^*)$ (cf. also \cite{haj2016sharp}).
Throughout this paper we will use the symbol $c_p$ to denote the best constant of the aforementioned embedding $X_0\hookrightarrow L^p(\R^n)$, namely
$$
c_p:=\sup_{x\in X_0\setminus\{0\}}\frac{\left\|u\right\|_{L^p(\R^n)}}{\left\|u\right\|} \quad \mbox{ for all } p\in[1,2^*],
$$
and we reserve the symbol $S$ to the constant
\begin{equation}\label{S}
S:=\inf_{u\in X_0\setminus\{0\}}\frac{\left\| u\right\|^2}{\left\|u\right\|_{L^{2^*}(\Omega)}^2} = \inf_{u\in X_0\setminus\{0\}}\frac{\left\| u\right\|^2}{\left\|u\right\|_{L^{2^*}(\R^n)}^2}.
\end{equation}

The symbols $\kappa_j$, $j=1,2,\ldots$, indicate generic positive constants, whose value may change from appearence to appearance.

Wih the above premises, it is straightforward to prove that \eqref{problema} is the Euler-Lagrange equation of the $C^1$ functional $\mathcal{E}_{\lambda,\mu}:X_0\to \R$ defined by
\begin{equation}\label{funzenergia}
\mathcal{E}_{\lambda,\mu}(u):=\frac{1}{2}\left\|u\right\|^2 -\frac{\mu}{2^*}\int_\Omega |u|^{2^*}dx  -\lambda\int_\Omega G(u)dx
\end{equation}
for every $u\in X_0$, where
$$
G(t):=\int_0^t g(\tau)d\tau \quad \mbox{ for all } t\in\R.
$$
As a result, the search for weak solutions to problem \eqref{problema}, namely functions $u\in X_0$ such that
$$
\left\langle u,v\right\rangle - \mu\int_\Omega |u|^{2^*-2}uv dx - \lambda \int_\Omega g(u)vdx =0
$$
for every $v\in X_0$, reduces to the search for critical points of the energy functional ${\mathcal E}_{\lambda,\mu}$.

%%%%%%%%%%%%%%%%%%%%%%%%%%%%%%%%%
\section{A direct minimization approach}\label{directminimization}

Our main result reads as follows:
\begin{theorem}\label{principale}
	Let $g:\R\to\R$ be a continuous function for which
	\begin{itemize}
		\item[$(g_1)$] there exist $a_1,a_2>0$ and $q\in[1,2^*)$ such that
		$$
		|g(t)|\leq a_1 + a_2|t|^{q-1} \quad  \mbox{ for all } t\in\R,
		$$
		\item[$(g_2)$] $\displaystyle\liminf_{t\to 0^+}\frac{G(t)}{t^2} = +\infty$.
	\end{itemize}
	Further, for any $\mu>0$ let $h_\mu:[0,+\infty)\to\R$ be the function defined by
	\begin{equation}\label{funzionehmu}
	h_\mu(t):=\frac{t-\mu c_{2^*}^{2^*} t^{2^*-1}}{a_1 c_{2^*}|\Omega|^\frac{2^*-1}{2^*} + a_2 c_{2^*}^q |\Omega|^\frac{2^*-q}{2^*} t^{q-1}} \quad  \mbox{ for all } t\geq 0.
	\end{equation}
	Then, for all $\mu>0$ there exists an open interval
	$$
	\Lambda _\mu \subseteq \left(0,\max_{[0,+\infty)}h_\mu\right)
	$$
	such that, for every $\lambda\in\Lambda_\mu$, \eqref{problema} admits a non-trivial weak solution $u_{0,\mu,\lambda}\in X_0$.
\end{theorem}

As already mentioned, we will prove the existence of such a weak solution by showing that the energy functional \eqref{funzenergia} attains a non-zero local minimum in $X_0$. Yet, due to the presence of the critical term, the direct minimization is not immediately applicable but requires some preliminary results on the following functionals.

\begin{lemma}\label{lemmasemicontinuita}
For every $\mu>0$ there exists $\varrho_{0,\mu}>0$ such that the functional
$$
\widehat{\mathcal E}_\mu(u):=\frac{1}{2}\left\|u\right\|^2 -\frac{\mu}{2^*}\int_\Omega |u|^{2^*}dx, \quad u\in X_0,
$$
is sequentially weakly lower semicontinuous in $B_c(0,\varrho_{0,\mu})$.
\end{lemma}

\begin{proof}
Let $\mu,\varrho>0$ and let $\{u_j\}_{j\in \N}\subset B_c(0,\varrho)$ be a sequence weakly convergent to some $u_\infty\in B_c(0,\varrho)$.  We will prove that
\begin{equation}\label{convergenzeFine}
%\begin{split}
%&\liminf_{j\rightarrow\infty}(\widehat{\mathcal E}_\mu(u_j)-\widehat{\mathcal E}_\mu(u_\infty))\\
%&=\liminf_{j\rightarrow\infty}\left( \frac{1}{2}\left(\left\|u_j\right\|^2-\left\|u_\infty\right\|^2\right)  -\frac{\mu}{2^*}\left(\int_\Omega |u_j|^{2^*}dx - \int_\Omega |u_\infty|^{2^*}dx\right)\right)\geq 0.
%\end{split}
\liminf_{j\rightarrow\infty}(\widehat{\mathcal E}_\mu(u_j)-\widehat{\mathcal E}_\mu(u_\infty))\geq 0.
\end{equation}

The following basic algebraic inequality % (see \cite[Lemma 4.2]{lin1990on})
$$
|b|^2-|a|^2\geq 2a(b-a) + |a-b|^2, \quad a,b\in\R,
$$
used with
$$
a:=u_\infty(x)-u_\infty(y), \quad b:=u_j(x)-u_j(y), \quad x,y\in\R^n,
$$
produces
\begin{equation}\label{GMB1}
\begin{split}
 & \frac{1}{2}\int_Q \left(|u_j(x)-u_j(y)|^2 - |u_\infty(x)-u_\infty(y)|^2\right) K(x-y)dxdy\\
 & \geq \int_Q (u_\infty(x)-u_\infty(y))(u_j(x)-u_j(y) - u_\infty(x) + u_\infty(y))K(x-y)dxdy\\
 & +\frac{1}{2}\int_Q |u_j(x)-u_j(y)-u_\infty(x)+u_\infty(y)|^2 K(x-y)dxdy.
\end{split}
\end{equation}

On the other hand, due to Br\'{e}zis-Lieb's Lemma, we obtain
\begin{equation}\label{GMB2}
\liminf_{j\rightarrow \infty}\left( \int_\Omega |u_j(x)|^{2^*}dx - \int_\Omega |u_\infty(x)|^{2^*}dx\right) = \liminf_{j\rightarrow \infty}\int_\Omega |u_j(x)-u_\infty(x)|^{2^*}dx
\end{equation}
while the convergence $u_j\rightharpoonup u_\infty$ implies that
%Further, testing \eqref{convergenze000} with $v=u_\infty$, it follows that
\begin{equation}\label{convergenze0000}
\lim_{j\to\infty}\int_Q (u_\infty(x)-u_\infty(y))(u_j(x)-u_j(y) - u_\infty(x) + u_\infty(y))K(x-y)dxdy = 0.
\end{equation}

%Hence, by using \eqref{GMB1}, \eqref{GMB2} and \eqref{convergenze0000} one has
Hence
\begin{equation}\label{perGMB22}
\begin{split}
\liminf_{j\rightarrow \infty}(\widehat{\mathcal E}_\mu(u_j)-\widehat{\mathcal E}_\mu(u_\infty)) &\geq \liminf_{j\rightarrow \infty}\left( \frac{1}{2}\left\|u_j-u_\infty\right\|^2 -\frac{\mu}{2^*}\int_\Omega |u_j(x)-u_\infty(x)|^{2^*}dx\right)\\
& \geq \liminf_{j\rightarrow \infty} \left(\frac{1}{2}-\frac{c_{2^*}^{2^*}}{2^*}\mu\|u_j-u_\infty\|^{2^*-2}\right)\|u_j-u_\infty\|^{2}\\
& \geq \liminf_{j\rightarrow \infty}\left(\frac{1}{2} - \frac{c_{2^*}^{2^*}}{2^*}\mu\varrho^{2^*-2}\right) \|u_j-u_\infty\|^{2}.
\end{split}
\end{equation}
So for $\varrho$ sufficiently small, that is
\begin{equation}\label{rhosegnatomu}
0<\varrho\leq \bar{\varrho}_\mu:=\left(\frac{n}{(n-2s) c_{2^*}^{2^*}\mu}\right)^{\frac{n-2s}{4s}},
\end{equation}
inequality \eqref{convergenzeFine} is verified and the functional $\widehat{\mathcal E}_\mu$ is sequentially weakly lower semicontinuous in $B_c(0,\varrho_{0,\mu})$, provided that $\varrho_{0,\mu}\in (0,\bar{\varrho}_\mu)$.
\end{proof}

\begin{lemma}\label{proprietaEtilde}
Let $\lambda,\mu>0$, let $g:\R\to\R$ satisfy $(g_1)$, and let $\widetilde{\mathcal E}_{\lambda,\mu}:X_0\to\R$ be the functional defined by
$$
\widetilde{\mathcal E}_{\lambda,\mu}(u):=\frac{\mu}{2^*}\int_\Omega |u|^{2^*}dx + \lambda \int_\Omega G(u)dx
$$
for any $u\in X_0$. Then the following facts holds:
\begin{itemize}
\item[$(i)$] If
\begin{equation}\label{Ga1}
\limsup_{\varepsilon\rightarrow 0^+}\frac{\displaystyle\sup_{B_c(0,\varrho_0)}\widetilde{\mathcal E}_{\lambda,\mu} - \sup_{B_c(0,\varrho_0-\epsilon)}\widetilde{\mathcal E}_{\lambda,\mu}}{\varepsilon}<\varrho_0
\end{equation}
for some $\varrho_0>0$, then
\begin{equation}\label{Ga2}
\inf_{\sigma<\varrho_0}\frac{\displaystyle \sup_{B_c(0,\varrho_0)}\widetilde{\mathcal E}_{\lambda,\mu} - \sup_{B_c(0,\sigma)}\widetilde{\mathcal E}_{\lambda,\mu}}{\varrho_0^2-\sigma^2}<\frac{1}{2};
\end{equation}
\item[$(ii)$] if \eqref{Ga2} is satisfied for some $\varrho_0>0$, then
\begin{equation}\label{VGa1}
\inf_{u\in B(0,\varrho_0)}\frac{\displaystyle \sup_{B_c(0,\varrho_0)}\widetilde{\mathcal E}_{\lambda,\mu} - \widetilde{\mathcal E}_{\lambda,\mu}(u)}{\varrho_0^2-\left\|u\right\|^2}<\frac{1}{2}.
\end{equation}
\end{itemize}
\end{lemma}

\begin{proof}
\emph{$(i)$} Observe first that if \eqref{Ga1} is fulfilled for some $\varrho_0>0$, then
\begin{equation}\label{Ga1GG}
\limsup_{\varepsilon\rightarrow 0^+}\frac{\displaystyle\sup_{B_c(0,\varrho_0)}\widetilde{\mathcal E}_{\lambda,\mu} - \sup_{B_c(0,\varrho_0-\epsilon)}\widetilde{\mathcal E}_{\lambda,\mu}}{\varrho_0^2 -(\varrho_0-\varepsilon)^2}<\frac{1}{2},
\end{equation}
as one can easily see considering that
\begin{equation}
\frac{\displaystyle\sup_{B_c(0,\varrho_0)}\widetilde{\mathcal E}_{\lambda,\mu} - \sup_{B_c(0,\varrho_0-\epsilon)}\widetilde{\mathcal E}_{\lambda,\mu}}{\varrho_0^2 -(\varrho_0-\varepsilon)^2} =\left(  \frac{\displaystyle\sup_{B_c(0,\varrho_0)}\widetilde{\mathcal E}_{\lambda,\mu} - \sup_{B_c(0,\varrho_0-\epsilon)}\widetilde{\mathcal E}_{\lambda,\mu}}{\varepsilon}\right) \left( \frac{1}{2\varrho_0-\varepsilon}\right) .
\end{equation}

Now, by \eqref{Ga1GG} there exists $\bar\varepsilon_0>0$ such that
$$
\frac{\displaystyle\sup_{B_c(0,\varrho_0)}\widetilde{\mathcal E}_{\lambda,\mu} - \sup_{B_c(0,\varrho_0-\varepsilon)}\widetilde{\mathcal E}_{\lambda,\mu}}{\varrho_0^2 -(\varrho_0-\varepsilon)^2}<\frac{1}{2}
$$
for every $\varepsilon\in (0,\bar\varepsilon_0)$; setting $\sigma_0:=\varrho_0-\varepsilon_0$, with $\varepsilon_0\in (0,\bar\varepsilon_0)$, it follows that
$$
\frac{\displaystyle\sup_{B_c(0,\varrho_0)}\widetilde{\mathcal E}_{\lambda,\mu} - \sup_{B_c(0,\sigma_0)}\widetilde{\mathcal E}_{\lambda,\mu}}{\varrho_0^2 -\sigma_0^2}<\frac{1}{2}
$$
and thus the conclusion follows.

\emph{$(ii)$} Thanks to inequality \eqref{Ga2} one has
\begin{equation}\label{Ga123}
\sup_{B_c(0,\sigma_0)}\widetilde{\mathcal E}_{\lambda,\mu} > \sup_{B_c(0,\varrho_0)}\widetilde{\mathcal E}_{\lambda,\mu} -\frac{1}{2}(\varrho_0^2-\sigma^2_0)
\end{equation}
for some $0<\sigma_0<\varrho_0$. Assumption $(g_1)$ and standard arguments show that $\widetilde{\mathcal E}_{\lambda,\mu}$ is weakly lower semicontinuous in $B_c(0,\sigma_0)$ and therefore one has
%in $B_c(0,\sigma_0)$ weak and sequential weak lower semicontinuity of $\widetilde{\mathcal E}_{\lambda,\mu}$ agree
\begin{equation*}\label{Ga1234}
\sup_{\partial B_c(0,\sigma_0)}\widetilde{\mathcal E}_{\lambda,\mu} = \sup_{\overline{\partial B_c(0,\sigma_0)}^w}\widetilde{\mathcal E}_{\lambda,\mu} = \sup_{B_c(0,\sigma_0)}\widetilde{\mathcal E}_{\lambda,\mu},
\end{equation*}
$\overline{\partial B_c(0,\sigma_0)}^w$ being the weak closure of $\partial B_c(0,\sigma_0)$ in $X_0$, and by \eqref{Ga123} there exists $u_0\in X_0$ with $\|u_0\|=\sigma_0$ such that
%\begin{equation}\label{Ga1235}
$$
\widetilde{\mathcal E}_{\lambda,\mu}(u_0)>\sup_{B_c(0,\sigma_0)}\widetilde{\mathcal E}_{\lambda,\mu}-\frac{1}{2}(\varrho_0^2-\sigma^2_0),
$$
that is,
$$
\frac{\displaystyle\sup_{B_c(0,\varrho_0)}\widetilde{\mathcal E}_{\lambda,\mu} - \widetilde{\mathcal E}_{\lambda,\mu}(u_0)}{\varrho_0^2- \|u_0\|^2}<\frac{1}{2},
$$
%\end{equation}
and the second claim is proved as well.
\end{proof}

\smallskip
We are now in position to prove our existence result.
\smallskip

\emph{Proof of Theorem \ref{principale}.} Fix $\mu>0$ and let $\varrho_{\mu,\max}>0$ be the point of global maximum of $h_\mu$; set $\varrho_{0,\mu}:=\min\{\bar{\varrho}_\mu, \varrho_{\mu,\max}\}$, $\bar{\varrho}_\mu$ being defined by \eqref{rhosegnatomu}, and take
$$
\lambda\in\Lambda_\mu:=(0,h_\mu(\varrho_{0,\mu})).
$$

As a result, there exists $\varrho_{0,\mu,\lambda}\in(0,\varrho_{0,\mu})$ such that
\begin{equation}\label{rangelambda}
0<\lambda<\frac{\varrho_{0,\mu,\lambda}-\mu c_{2^*}^{2^*}\varrho_{0,\mu,\lambda}^{2^*-1}}{ac_{2^*}|\Omega|^\frac{2^*-1}{2^*} + ac_{2^*}^q |\Omega|^\frac{2^*-q}{2^*}\varrho_{0,\mu,\lambda}^{q-1}}.	
\end{equation}

Since $\varrho_{0,\mu,\lambda}<\bar{\varrho}_\mu$, by Lemma \ref{lemmasemicontinuita} the functional $\mathcal E_{\lambda,\mu}$ is sequentially weakly lower semicontinuous in $B_c(0,\varrho_{0,\mu,\lambda})$ and then possesses a global minimum $u_{0,\mu,\lambda}$ in it.

Arguing by contradiction, suppose that $\|u_{0,\mu,\lambda}\|=\varrho_{0,\mu,\lambda}$. Let $0<\varepsilon<\varrho_{0,\mu,\lambda}$ and set
$$
\varphi_{\lambda,\mu}(\varepsilon,\varrho_{0,\mu,\lambda}):=\frac{\displaystyle \sup_{B_c(0,\varrho_{0,\mu,\lambda})}\widetilde{\mathcal E}_{\lambda,\mu} - \sup_{B_c(0,\varrho_{0,\mu,\lambda}-\varepsilon)}\widetilde{\mathcal E}_{\lambda,\mu}}{\varepsilon}.
$$

By $(g_1)$ one has
\begin{align*}
\varphi_{\lambda,\mu}(\varepsilon,\varrho_{0,\mu,\lambda}) &\leq\frac{1}{\varepsilon}\sup_{u\in B_c(0,1)}\int_\Omega\left|\int_{(\varrho_{0,\mu,\lambda}-\varepsilon)u(x)}^{\varrho_{0,\mu,\lambda} u(x)}\left(\mu|t|^{2^*-1} +\lambda |g(t)|\right) dt\right|dx\\
&\leq \frac{1}{\varepsilon}\sup_{u\in B_c(0,1)}\int_\Omega\left|\int_{(\varrho_{0,\mu,\lambda}-\varepsilon)u(x)}^{\varrho_{0,\mu,\lambda} u(x)}\left(\mu|t|^{2^*-1} + a_1\lambda + a_2\lambda |t|^{q-1}\right) dt\right|dx\\
& \leq \frac{c_{2^*}^{2^*}\mu}{2^*}\left(\frac{\varrho_{0,\mu,\lambda}^{2^*}-(\varrho_{0,\mu,\lambda}-\varepsilon)^{2^*}}{\varepsilon}\right) + a_1\lambda c_{2^*}|\Omega|^\frac{2^*-1}{2^*} \\
&\;\;\; + a_2 \lambda \frac{c_{2^*}^q}{q}|\Omega|^\frac{2^*-q}{2^*}\left(\frac{\varrho_{0,\mu,\lambda}^q-(\varrho_{0,\mu,\lambda}-\varepsilon)^q}{\varepsilon}\right)
\end{align*}
and passing to the limsup for $\varepsilon\rightarrow 0^+$ we get
\begin{equation}\label{VGa1Glu1}
\limsup_{\varepsilon\rightarrow 0^+}\varphi_{\lambda,\mu}(\varepsilon,\varrho_{0,\mu,\lambda}) < c_{2^*}^{2^*}\mu \varrho_{0,\mu,\lambda}^{2^*-1} + \lambda a_1 c_{2^*}|\Omega|^\frac{2^*-1}{2^*} + \lambda a_2 c_{2^*}^q|\Omega|^\frac{2^*-q}{2^*}\varrho_{0,\mu,\lambda}^{q-1},
\end{equation}
which due to \eqref{rangelambda} forces
$$
\limsup_{\varepsilon\rightarrow 0^+}\varphi_{\lambda,\mu}(\varepsilon,\varrho_{0,\mu,\lambda})<\varrho_{0,\mu,\lambda}.
$$

So, by Lemma \ref{proprietaEtilde},
\begin{equation*}\label{VGa1Glu}
\inf_{u\in B(0,\varrho_{0,\mu,\lambda})}\frac{\displaystyle \sup_{B_c(0,\varrho_{0,\mu,\lambda})}\widetilde{\mathcal E}_{\lambda,\mu} - \widetilde{\mathcal E}_{\lambda,\mu}(u)}{\varrho_{0,\mu,\lambda}^2-\left\|u\right\|^2}<\frac{1}{2}
\end{equation*}
and therefore there exists $w_{\mu,\lambda}\in B(0,\varrho_{0,\mu,\lambda})$ such that
$$
\widetilde{\mathcal E}_{\lambda,\mu}(u)\leq \sup_{B_c(0,\varrho_{0,\mu,\lambda})}\widetilde{\mathcal E}_{\lambda,\mu} < \widetilde{\mathcal E}_{\lambda,\mu}(w_{\mu,\lambda}) + \frac{1}{2}(\varrho_{0,\mu,\lambda}^2-\|w_{\mu,\lambda}\|^2),
$$
for every $u\in B_c(0,\varrho_{0,\mu,\lambda})$. Thus
\begin{equation}\label{John}
\mathcal E_{\lambda,\mu}(w_{\mu,\lambda}):=\frac{1}{2}\|w_{\mu,\lambda}\|^2 - \widetilde{\mathcal E}_{\lambda,\mu}(w_{\mu,\lambda})<\frac{\varrho_{0,\mu,\lambda}^2}{2}-\widetilde{\mathcal E}_{\lambda,\mu}(u)
\end{equation}
for every $u\in B_c(0,\varrho_{0,\mu,\lambda})$ and, on the other hand,
%Since $\varrho_{0,\mu,\lambda}<\bar{\varrho}_\mu$, the functional $\mathcal E_{\lambda,\mu}$ is sequentially weakly lower semicontinuous in $B_c(0,\varrho_{0,\mu,\lambda})$ and then it possesses a global minimum $u_{0,\mu,\lambda}$ in it. If $\|u_{0,\mu,\lambda}\|=\varrho_{0,\mu,\lambda}$
%by \eqref{John} one would have
\begin{equation*}
\mathcal E_{\lambda,\mu}(u_{0,\mu,\lambda})=\frac{1}{2}\varrho_{0,\mu,\lambda}^2 - \widetilde{\mathcal E}_{\lambda,\mu}(u_{0,\mu,\lambda})>\mathcal E_{\lambda,\mu}(w_{\mu,\lambda}),
\end{equation*}
which is a contradiction. In conclusion $u_{0,\mu,\lambda}$ is located inside $B_c(0,\varrho_{0,\mu,\lambda})$ and is therefore a local minimum for $\mathcal E_{\lambda,\mu}$ and a weak solution to \eqref{problema}.

The last issue left open is to show that $u_{0,\mu,\lambda}$ is not identically $0$ in $\Omega$. To this end, fix $\bar{u}\in C_0^2(\Omega)$, $\bar{u}\geq 0$, $\bar{u}\not\equiv 0$ in $\Omega$. Thanks to $(g_2)$, for every $M>0$ there exists $\delta>0$ so that
$$
G(t)\geq M t^2 \quad \mbox{for any } t\in(0,\delta).
$$
So if $t\in\left(0,\delta/\sup_{\Omega}\bar{u}\right)$, we obtain
\begin{align*}
\mathcal E_{\lambda,\mu}(t\bar{u}) & = \frac{1}{2}\left\|\bar{u}\right\|^2 t^2 -\frac{\mu}{2^*}\int_\Omega |\bar{u}|^{2^*} dx\; t^{2^*} - \lambda\int_\Omega G(t\bar{u})dx\\
& \leq  \left( \frac{1}{2}\left\|\bar{u}\right\|^2 -\lambda M \left\|\bar{u}\right\|^2_{L^2(\Omega)}\right) t^2 -\frac{\mu}{2^*}\int_\Omega |\bar{u}|^{2^*} dx\; t^{2^*}\\
&<0
\end{align*}
for $M>0$ big enough. Therefore $0$ is not a local minimum point for $\mathcal E_{\lambda,\mu}$ and $u_{0,\mu,\lambda}$ is non-trivial.
\qed

\begin{remark}
	\rm{
As one can deduce from the above proof, assumption $(g_2)$ only prevents $0$ from being a local minimum point for $\mathcal E_{\lambda,\mu}$. Therefore it can be replaced by any other assumption, compatible with $(g_1)$, which guarantees this fact. For instance, if $g(0)\neq 0$ then $0$ is not a solution to \eqref{problema} and so $u_{0,\lambda,\mu}$ is non-trivial.
}
\end{remark}

As already pointed out in the introduction, the reasoning illustrated before can be also performed when the perturbation affects also a term singular at zero. More precisely, consider the following singular variant of problem \eqref{problema}:
\begin{equation}\tag{$\widetilde{P}_{\lambda,\mu}$}\label{problemasingolare}
	\left\{
	\begin{array}{ll}
		-\mathcal L_K u = \mu u^{2^*-1} +\lambda \left( u^{r-1} + g(u)\right)  & \mbox{ in } \Omega\\
		u>0 & \mbox{ in } \Omega\\
		u=0 & \mbox{ in } \R^n\setminus \Omega,
	\end{array} \right.
\end{equation}
where $r\in(0,1)$ and $g:[0,+\infty)\to[0,+\infty)$ is continuous and subcritical. By analyzing the profile of an auxiliary rational function similar to \eqref{funzionehmu} we are able to assert that, for all $\mu$'s and for small values of $\lambda$, problem \eqref{problemasingolare} admits a weak solution in $X_0$.

In the context of singular problems we recall that a weak solution to \eqref{problemasingolare} is understood to be any $u\in X_0$ such that
\begin{itemize}
	\item $u>0$ a.e. in $\Omega$,
	\item $u^{r-1}v\in L^1(\Omega)$ for any $v\in X_0$,
	\item $\left\langle u,v\right\rangle -\displaystyle\mu\int_\Omega u^{2^*-1}v dx -\lambda\displaystyle\int_\Omega \left( u^{r-1} + g(u)\right) vdx =0$ for each $v\in X_0$.
\end{itemize}

The energy associated with \eqref{problemasingolare} turns out to be
$$
\mathcal E_{\lambda,\mu}(u):=\frac{1}{2}\left\|u\right\|^2 -\frac{\mu}{2^*}\int_\Omega (u^+)^{2^*}dx -\frac{\lambda}{r}\int_\Omega (u^+)^r dx -\lambda\int_\Omega G(u^+) dx,
$$
for all $u\in X_0$, where $u^+:=\max\left\lbrace u,0 \right\rbrace$.

In the wake of Theorem \ref{principale} we can prove the following
\begin{theorem}\label{princsingolare}
	Let $g:[0,+\infty)\to[0,+\infty)$ be a continuous function for which
	\begin{itemize}
		\item[$(g_1')$] there exist $a_1,a_2\geq 0$ and $q\in[1,2^*)$ such that
		$$
		g(t)\leq a_1 + a_2 t^{q-1} \quad  \mbox{ for all } t\geq 0.
		$$
	\end{itemize}
	Further, for any $\mu>0$ let $k_\mu:[0,+\infty)\to\R$ be the function defined by
	$$
	k_\mu (t):=\frac{t^{2-r}-\mu c_{2^*}^{2^*} t^{2^*-r}}{ c_{2^*}^r |\Omega|^\frac{2^*-r}{2^*} + a_1 c_{2^*}|\Omega|^\frac{2^*-1}{2^*} t^{1-r} + a_2 c_{2^*}^q |\Omega|^\frac{2^*-q}{2^*} t^{q-r}} \quad  \mbox{ for all } t\geq 0.
	$$
	Then, for all $\mu>0$ there exists an open interval
	$$
	\Lambda _\mu \subseteq \left(0,\max_{[0,+\infty)}k_\mu\right)
	$$
	such that, for every $\lambda\in\Lambda_\mu$, \eqref{problemasingolare} admits a weak solution $\tilde{u}_{0,\lambda,\mu}\in X_0$.
\end{theorem}

\begin{proof}
	It is clear that for all $\lambda,\mu>0$ and $u\in X_0$ the functionals
	\begin{align*}
		\widehat{\mathcal E}_\mu(u) &:=\frac{1}{2}\left\|u\right\|^2 -\frac{\mu}{2^*}\int_\Omega (u^+)^{2^*} dx,\\
		\widetilde{\mathcal E}_{\lambda,\mu}(u)	&:=\frac{\mu}{2^*}\int_\Omega (u^+)^{2^*} dx +\frac{\lambda}{r}\int_\Omega (u^+)^r dx +\lambda\int_\Omega G(u^+) dx
	\end{align*}
	fulfill	Lemmas \ref{lemmasemicontinuita} and \ref{proprietaEtilde}, respectively.

	Arguing exactly as in Theorem \ref{principale} we deduce that $\mathcal E_{\lambda,\mu}$ admits a minimum $\tilde{u}_{0,\lambda,\mu}$ inside a sufficiently small ball of $X_0$. It is easily seen that such a minimum does not coincide with $0$; indeed fixing $u\in X_0$, $u>0$ in $\Omega$, if $t>0$ one has
	\begin{align*}
	\mathcal E_{\lambda,\mu}(tu) &\leq \frac{1}{2}\left\| u\right\|^2 t^2 - \frac{\mu}{2^*}\left\| u\right\|_{L^{2^*}(\Omega)}^{2^*}t^{2^*} +\lambda a_1\left\|u\right\|_{L^1(\Omega)} t \\
	& \quad +\frac{\lambda a_2}{q}\left\| u\right\|_{L^q(\Omega)}^q t^q -\lambda\frac{t^r}{r}\int_\Omega |u|^r dx, \;
	\end{align*}
	and hence $\mathcal E_{\lambda,\mu}(tu)$ is negative for $t$ small enough. The fact that $\tilde{u}_{0,\lambda,\mu}$ is a weak solution to \eqref{problemasingolare} can be inferred in a vein similar to \cite{farfar2015a}.
\end{proof}

\begin{example}
	\rm{A simple application of Theorem \ref{princsingolare} is the coming one. Let $n=3$, $\Omega = B\left( 0,(3/(4\pi))^{1/3}\right) \subset\R^3$, $g\equiv 0$ and $r=1/2$.

Then, the study of the function
$$
k_\mu (t)=c_{2^*}^{-\frac{1}{2}}\left(t^\frac{3}{2}-\mu c_{2^*}^{2^*}t^\frac{2\cdot 2^* -1}{2}\right), \quad t\geq 0,
$$
where $2^*=6/(3-2s)$, allows us to deduce that for any $\mu>0$ there exists an open interval
$$
\Lambda_\mu \subseteq \left( 0, c_{2^*}^{-\frac{1}{2}}\left( \left(\frac{3}{(2\cdot 2^* -1)\mu c_{2^*}^{2^*}}\right)^\frac{3}{2(2^*-2)}-\mu c_{2^*}^{2^*}\left(\frac{3}{(2\cdot 2^* -1)\mu c_{2^*}^{2^*}}\right)^\frac{2\cdot 2^* -1}{2(2^*-2)}\right)\right)
$$
such that for each $\lambda\in\Lambda_\mu$, the problem
$$
\left\{
\begin{array}{ll}
-\mathcal L_K u = \mu u^{2^*-1} + \lambda u^{-\frac{1}{2}}  & \mbox{ in } \Omega\\
u=0 & \mbox{ in } \R^n\setminus \Omega,
\end{array} \right.
$$	
admits a weak solution.
}
\end{example}

\begin{remark}\label{Costante}
\rm{
A crucial step in our approach is the explicit computation of the critical embedding constants that naturally appear in Theorems \ref{principale} and \ref{princsingolare}, as well as in their consequences (Theorems \ref{casoparticolare} and \ref{casoparticolaresingolare} in Introduction). We notice that, in the fractional Laplacian setting, following Cotsiolis and Tavoularis in \cite[Theorem 1.1]{cottav2004best}, one has
$$
\|u\|_{2^*_s}\leq 2^{-s}\pi^{-s/2}\left[\frac{\Gamma\left(\displaystyle\frac{n-2s}{2}\right)}{\Gamma\left(\displaystyle\frac{n+2s}{2}\right)}\right]^{1/2}\left[\frac{\Gamma(n)}{\Gamma\left(\displaystyle\frac{n}{2}\right)}\right]^{s/n}
\left\| u\right\|_{X_{0,s}},\quad\,\forall u\in X_{0,s}
$$
where
$$
\Gamma(\tau):=\int_0^{+\infty}z^{\tau-1}e^{-z}dz, \quad\forall \tau>0.
$$
}
\end{remark}
	
%%%%%%%%%%%%%%%%%%%%%%%%%%%%%%%%%%%%%%%%%%%%%%%%%%%
\section{A mountain-pass approach}\label{mountainpass}
In this short section we show that conveniently modifying the assumptions on our data it is possible, for small values of $\mu$ and arbitrary ones of $\lambda$, to deduce the existence of one non-trivial non-negative solution of mountain-pass-type, retrieving the approach of \cite{serval2015the}.

As usual in the presence of critical nonlinearities, the key point is that the "altitude" of the mountain pass lies under a certain threshold involving the constant \eqref{S} and this indeed occurs for small $\mu$ (see assumption $(H_1)$ below, that for $\mu=1$ gives back assumption (1.25) of \cite[Theorem 1]{serval2015the}).

\begin{theorem}\label{secondoteorema}
Let $g:\R\to\R$ be a continuous function satisfying $(g_1)$ and
\begin{itemize}
\item[$(g_3)$] $\displaystyle\lim_{t\to 0}\frac{g(t)}{t}=0$.
\end{itemize}
Furthermore for any $\mu\in\left(0,n/(n-2s) \right)$ assume that there exists $u^\star\in X_0\setminus\{0\}$, $u^\star\geq 0$ a.e. in $\R^n$, such that
\begin{itemize}
\item[$(H_1)$] $\displaystyle\sup_{\tau\geq 0}\mathcal E_{\lambda,\mu}(\tau u^\star)<\frac{n-\mu(n-2s)}{2n}\mu^{-\frac{n-2s}{2s}}S^\frac{n}{2s}$.	
\end{itemize}
Then for all $\lambda>0$ problem \eqref{problema} admits a non-trivial non-negative weak solution.	
\end{theorem}

\begin{proof}
The proof relies upon the same arguments as \cite{serval2015the}; we only sketch the most notable steps and point out the differences.

Fix $\mu\in\left(0,n/(n-2s) \right)$ and $\lambda\in(0,+\infty)$. Assumptions $(g_1)$ and $(g_3)$, together with $(H_1)$, allow us to infer the mountain pass geometry for $\mathcal E_{\lambda,\mu}$, i.e. that
\begin{itemize}
\item[$(i)$] there exist $\varrho>0$ and $\alpha>0$ such that for any $u\in \partial B(0,\varrho)$ one has $\mathcal E_{\lambda,\mu}(u)\geq\alpha$; \item[$(ii)$] there exists $u_1\in X_0$ such that $u_1\geq 0$ a.e. in $\R^n$,
$\|u_1\|>\varrho$ and $\mathcal E_{\lambda,\mu}(u_1)<\alpha$.
\end{itemize}
Setting
\begin{equation}\label{quotaMP}
c:=\inf_{\gamma\in \Gamma}\sup_{t\in[0,1]}\mathcal E_{\lambda,\mu}(\gamma(t)),
\end{equation}
where
$$
\Gamma:=\{\gamma\in C^0([0,1], X_0) : \gamma(0)=0 \mbox{ and } \gamma(1)=u_1\},
$$
on account of $(H_1)$ one has the estimate
\begin{equation}\label{relazionec}
\alpha \leq c < \frac{n-\mu(n-2s)}{2n}\mu^{-\frac{n-2s}{2s}} S^\frac{n}{2s}.
\end{equation}

Now, Theorem 2.2 of \cite{brenir1983positive} assures the existence of a Palais-Smale sequence $\{u_j\}_{j\in\N}\subset X_0$ at the level $c$ for $\mathcal E_{\lambda,\mu}$, namely a sequence fulfilling
\begin{equation}\label{quasipuntocritico}
\mathcal E_{\lambda,\mu}(u_j)\to c, \quad \mathcal E'_{\lambda,\mu}(u_j)\to 0, \quad \mbox{as } j\to\infty
\end{equation}
and, in view of $(g_1)$, this is bounded in $X_0$. So, up to a subsequence still denoted by $\{u_j\}_{j\in\N}$, there exists $u_\infty \in X_0$
such that $u_j \rightharpoonup u_\infty$, that is
\begin{equation}\label{convergenze0}
\left\langle u_j,v\right\rangle  \to \left\langle u_\infty,v\right\rangle \quad \mbox{as } j\to\infty
\end{equation}
for any $v\in X_0$. The fact that $\{u_j\}_{j\in\N}$ is bounded in $L^{2^*}(\Omega)$ as well implies that, passing to a further subsequence,
\begin{equation}\label{convergenze}
\begin{split}
& u_j \rightharpoonup u_\infty \quad \mbox{ in } L^{2^*}(\R^n)\\
& u_j \to u_\infty \quad \mbox{ in } L^p(\R^n), \quad p\in[1,2^*) \\
& u_j \to u_\infty \quad \mbox{ a.e. in } \R^n\\
& |u_j|^{2^*-2}u_j \rightharpoonup  |u_\infty|^{2^*-2}u_\infty \quad \mbox{ in } L^\frac{2^*}{2^*-1}(\Omega)
\end{split}
\end{equation}
as $j\to\infty$. Moreover, assumptions $(g_1), (g_3)$ tell us that for every $\varepsilon>0$ there exists $\kappa_\varepsilon>0$ such that the inequality
\begin{equation}\label{disuggduestar}
|g(t)|\leq \varepsilon |t|^{2^*-1} + \kappa_\varepsilon
\end{equation}
holds for every $t\in\R$ and, together with \eqref{convergenze}, this leads to
\begin{align*}\label{convfae}
& g(u_j(\cdot)) \to  g(u_\infty(\cdot)) \quad \mbox{ a.e. in } \Omega\\
& g(u_j(\cdot)) \rightharpoonup  g(u_\infty(\cdot)) \quad \mbox{ in }  L^\frac{2^*}{2^*-1}(\Omega),
\end{align*}
as $j\to\infty$. As a result,
\begin{equation}\label{convffi1}
\int_\Omega g(u_j)v dx \to  \int_\Omega g(u_\infty)v dx \quad \mbox{as } j\to\infty
\end{equation}
for any $v\in X_0$. Now, recalling that
$$
\left\langle \mathcal E'_{\lambda,\mu}(u_j),v\right\rangle = \left\langle u_j,v\right\rangle  - \mu\int_\Omega |u_j|^{2^*-2}u_j v dx - \lambda \int_\Omega g(u_j) v dx
$$
for any $v\in X_0$, passing to the limit as $j\to\infty$ into the above equality and taking  \eqref{quasipuntocritico}, \eqref{convergenze0}, \eqref{convergenze} and \eqref{convffi1} into account, we get
$$
\left\langle u_\infty,v\right\rangle - \mu\int_\Omega |u_\infty|^{2^*-2}u_\infty v dx - \lambda \int_\Omega g(u_\infty) v dx = 0
$$
for any $v \in X_0$, so $u_\infty$ is a weak solution to \eqref{problema}.

Finally, suppose that $u_\infty \equiv 0$ in $\Omega$. Then, due to  \eqref{convergenze} and \eqref{disuggduestar}, we would obtain that
\begin{equation}\label{fujlim}
\lim_{j\to\infty}\int_\Omega g(u_j) u_j dx = \lim_{j\to\infty}\int_\Omega G(u_j) dx = 0
\end{equation}
and therefore, since $\left\langle \mathcal E'_{\lambda,\mu}(u_j), u_j\right\rangle \to 0$ as $j\to\infty$, %thanks to \eqref{convergenze} and \eqref{fujlim},
that
\begin{equation}\label{nonso}
\left\| u_j\right\|^2  - \mu\int_\Omega |u_j|^{2^*} dx \to 0 \quad \mbox{as } j\to\infty.
\end{equation}
Now, the boundedness of $\{\|u_j\|\}_{j\in\N}$ in $\R$ implies that, up to a subsequence, there exists $L\in[0,+\infty)$ such that
\begin{equation}\label{L1}
\|u_j\|^2 \to L
\end{equation}
and so
\begin{equation}\label{L2}
\int_\Omega |u_j|^{2^*} dx\to \frac{L}{\mu}
\end{equation}
as $j\to +\infty$. Recalling that
$$
\mathcal E_{\lambda,\mu}(u_j)=\frac 1 2\left\| u_j\right\|^2 -\frac{\mu}{2^*} \int_\Omega |u_j|^{2^*} dx -
\lambda\int_\Omega G(u_j)dx \to c \quad \mbox{as } j\to\infty,
$$
on account of \eqref{fujlim}, \eqref{L1} and \eqref{L2}, it follows that
\begin{equation}\label{c}
c=\left( \frac 1 2 -\frac{\mu}{2^*}\right) L=\frac{n-\mu(n-2s)}{2n} L,
\end{equation}
which being $c\geq \alpha>0$ forces $L>0$. On the other hand,
$$
\left\| u_j\right\|^2 \geq S \|u_j\|_{L^{2^*}(\Omega)}^2
$$
so that, passing to the limit as $j\to\infty$ and taking  \eqref{L1} and \eqref{L2} into account, we get
$$
L\geq S \left(\frac{L}{\mu}\right)^\frac{2}{2^*},
$$
which combined with \eqref{c} gives
$$
c\geq \frac{n-\mu(n-2s)}{2n} \mu^{-\frac{2}{2^*-2}} S^\frac{2^*}{2^*-2} = \frac{n-\mu(n-2s)}{2n} \mu^{-\frac{n-2s}{2s}} S^\frac{n}{2s}.
$$
This contradicts \eqref{relazionec} and therefore $u_\infty\not\equiv 0$ in $\Omega$. The proof is now complete.
\end{proof}

\noindent {\bf Acknowledgements.}
The manuscript was realized under the auspices of the
INdAM - GNAMPA Project 2016 titled: {\it Problemi variazionali su variet\`a Riemanniane e gruppi di Carnot}.

%%%%%%%%%%%%%%%%%%%%%%%%%%%%%%%%%%%%%%%%%%%%%%%%%%%%%%%%
%\nocite{haj2013on,haj2016a,haj2016sharp,haj2016symmetry}

%\bibliographystyle{plain}
%\bibliography{databiblio}
\end{document}